\documentclass[12pt,reqno]{article}

\usepackage[usenames,dvipsnames]{xcolor}
\usepackage{amssymb}
\usepackage{amsmath}
\usepackage{amsthm}
\usepackage{amsfonts}
\usepackage{amscd}
\usepackage{graphicx}

\usepackage[colorlinks=true,
linkcolor=webgreen,
filecolor=webbrown,
citecolor=webgreen]{hyperref}
\usepackage{cleveref}

\definecolor{webgreen}{rgb}{0,.5,0}
\definecolor{webbrown}{rgb}{.6,0,0}

\usepackage{fullpage}
\usepackage{float}

\usepackage{graphics}
\usepackage{latexsym}
\usepackage{epsfig}

\usepackage{enumitem}
\usepackage{pgf,tikz}
\usetikzlibrary{arrows, automata, calc, chains, decorations, shapes, positioning}

\newtheorem{lem}{Lemma}[section]
\newtheorem{theorem}{Theorem}[section]

\theoremstyle{definition}
\newtheorem{Def}{Definition}

\newtheorem{problem}{Problem}

\theoremstyle{remark}

\newtheorem{remark}{Remark}

\DeclareMathOperator{\Mod}{mod}
\DeclareMathOperator{\lcm}{lcm}
\DeclareMathOperator{\Card}{Card}
\newcommand{\tL}{\mathtt 1} 
\newcommand{\tO}{\mathtt 0}

\AtBeginDocument{%
   \def\MR#1{}
}

\makeatletter
\newcommand*{\rom}[1]{\expandafter\@slowromancap\romannumeral #1@}
\makeatother

\newcommand\nnfootnote[1]{%
  \begin{NoHyper}
  \renewcommand\thefootnote{}\footnote{#1}%
  \addtocounter{footnote}{-1}%
  \end{NoHyper}
}

\begin{document}

\begin{center}
\vskip 1cm{\LARGE\bf 
Correlation of the Rudin--Shapiro  \\
\vskip .1in
sequence along prime numbers
}
\vskip 1cm
\large
Pierre Popoli\\
Université de Lorraine, CNRS, IECL, F-54000 Nancy, France\\
\href{mailto: pierre.popoli@univ-lorraine.fr}{\tt pierre.popoli@univ-lorraine.fr}
\end{center}

\vskip .2 in

\begin{abstract}
The correlation measure of order $k$ is a measure of pseudorandomness that quantifies the similarity between a sequence and its shifts. It is known that the correlation of order $4$ is large for the Rudin--Shapiro sequence despite having nice pseudorandom properties with respect to the correlation of order $2$. In this paper, we prove a radically different behavior along the subsequence of prime numbers and continue the investigation towards the pseudorandomness of subsequences of automatic sequences. This result generalizes the result of Aloui, Mauduit, and Mkaouar (2021) about the correlation of the sum of digits function along prime numbers. 
\end{abstract}


\setcounter{tocdepth}{1}
\tableofcontents

\nnfootnote{\textbf{Keywords}: exponential sums, prime numbers, Rudin--Shapiro sequence, correlations.}

\section{Introduction}

\subsection{Notations}

In this paper, we write $\mathbb{N}=\{0,1,2,\ldots\}$, and the variable $p$ denotes a prime number. We write $e(x)=\exp(2\pi i x)$ for any real number $x$. For $n \in \mathbb{N}$, $v_p(n)$ stands for the $p$-adic valuation of $n$, i.e.\begin{align*} 
v_p(n)=\max \{\nu \in \mathbb{N} : p^{\nu}  \mid n\}.
\end{align*} We denote by $\Lambda$ the von Mangoldt function, that is $\Lambda(n)=\ln p$ if $n=p^{\nu}$ for some $\nu \geq 1$, and $\Lambda(n)=0$ if $n$ is not a prime power. We further denote by $\pi(x)$ the number of primes $p \leq x$ and $\pi(x;m,\ell)$ the number of prime $p\leq x$ that are congruent to $\ell \pmod m$ for some coprime pair $(\ell,m) \in \mathbb{N} \times \mathbb{N}$, $m\geq 1$. For $\alpha=(\alpha_0,\ldots,\alpha_k) \in \mathbb{R}^{k+1}$, we denote \begin{align*}
\tilde{\alpha}=\sum \limits_{j=1}^k j \alpha_j, \qquad \tilde{\alpha_i}=\sum \limits_{j=i}^k \alpha_j, \quad 0\leq i \leq k.
\end{align*} For two vectors $a,b  \in \mathbb{R}^k$, we denote by $a\cdot b=\sum_{i=1}^k a_ib_i$ the euclidean inner product. For any real $x$, we define $\lVert x \rVert$ as the distance from $x$ to the nearest integer. For an integer $M\geq 1$, we denote by $[M]$ the set $\{0,1,\ldots,M\}$. We denote $\log$ as the logarithm in base $10$ and $\log_b$ the logarithm in base $b$ for some integer $b\geq 2$. A binary word $\omega$ is a finite sequence of letters over the alphabet $\{\tO,\tL\}$, denoted $\omega=a_1\cdots a_k$ with $a_i \in \{\tO,\tL\}$. For $n\in \mathbb{N}$, we denote $(n)_2$ the binary word $\omega=\varepsilon_{\ell}\cdots\varepsilon_0$ such that $n=\sum_{0\leq i \leq \ell}\varepsilon_i2^{i}$. 

\subsection{Correlation measure}

The series of papers ``On finite pseudorandom binary sequences''~\rom{1}--\rom{7}, initiated by Mauduit and S\'{a}rk\"{o}zy in 1997, studies several measures of complexity for binary pseudorandom sequences. Among these measures of complexity, the correlation measure is one of the most important ones and studied in various contexts. In this section, we introduce the correlation measure of order $k$ and some theoretical results about this measure of complexity. 

\smallskip

Let $\mathcal{S}=(s_n)_{n\geq 0}$ be a sequence on $\{-1,1\}$ and let $\mathcal{S}_N$ be the finite sequence of its $N$ first terms. Let $N,k,M\geq 2$ be integers, $D=(d_1,\ldots,d_k) \in \mathbb{N}^k$ such that $d_1<d_2<\cdots<d_k$ and $M+d_k\leq N$. We denote $V(\mathcal{S}_N,M,D)$ the following sum, \[V(\mathcal{S}_N,M,D)=\sum \limits_{0\leq n \leq M-1}s_{n+d_1}s_{n+d_2}\cdots s_{n+d_k}.\] 

\begin{Def} \label{correlation_order_k}
Let $k\geq 2$ be an integer. The \emph{correlation measure of order $k$} of $\mathcal{S}=(s_n)_{n\geq 0}$, denoted $C_k(\mathcal{S},N)$, is defined by \[C_k(\mathcal{S},N)=\max \limits_{M,D} |V(\mathcal{S}_N,M,D)|=\max \limits_{M,D} \left| \sum \limits_{0\leq n \leq M-1}s_{n+d_1}s_{n+d_2}\cdots s_{n+d_k} \right|, \] where the maximum is taken over all $D=(d_1,\ldots,d_k)\in \mathbb{N}^k$ and $M$ such that $d_1<d_2<\cdots<d_k$ and $M+d_k \leq N$.
\end{Def} 

The correlation measure of order $k$ measures how many different values a sequence and its shifts have. If this is the case, the number of cancellations in the sum $V(\mathcal{S}_N,M,D)$ is small, and its correlation measure of order $k$ will be large. Thus, it is reasonable to think that the correlation measures of order $k$ for a random binary sequence will be small. Cassaigne, Mauduit, and S\'{a}rk\"{o}zy~\cite{CMS2001} have specified this postulate in the following theorem.

\begin{theorem}[\cite{CMS2001}]
For any $k \geq 2$ and for any real $\varepsilon>0$, there exist $N_0=N_0(\varepsilon,k)$ and $\delta=\delta(\varepsilon,k)>0$ such that for any $N \geq N_0$ we have with probability $>1-2\varepsilon$, \[\delta \sqrt{N} < C_k(\mathcal{S},N) < 5 \sqrt{kN\log N}. \]
\end{theorem}

Alon et al.~\cite{AKMMR2007} improved this result, particularly for the lower bound. 

\begin{theorem}[\cite{AKMMR2007}]
For any $N_0=N_0(\varepsilon)$, there exists a $N_0=N_0(\varepsilon)$ such that for any $N\geq N_0$, we have with probability $1-\varepsilon$, \[\frac{2}{5}\sqrt{N\log \binom{N}{k}}<C_k(\mathcal{S},N)<\frac{7}{4}\sqrt{N\log \binom{N}{k}}\] for any integer $k$ such that $2 \leq k \leq N/4$.
\end{theorem}

Therefore, if a sequence possesses a correlation measure of order $k$ larger (resp. smaller) than the expected order for a truly random binary sequence, i.e., $\sqrt{N\log N}$ up to a multiplicative constant, then we say that the correlation of this sequence is large (resp. small). 

\subsection{Rudin--Shapiro sequence}

The Rudin--Shapiro sequence $\mathcal{R}$, or Golay--Shapiro sequence, was first discovered by Golay~\cite{Golay1951} and was studied by Rudin~\cite{Rudin1959} building upon work by Shapiro~\cite{Shapiro1952}. Notice that the present paper considers the Rudin--Shapiro sequence on the binary alphabet $\{-1,1\}$. 

\begin{Def}
Let us denote by $\mathcal{R}=(r_n)_{n\geq 0}$ the Rudin--Shapiro, or Golay--Shapiro, sequence defined by \begin{align*}
r_n=(-1)^{r_{\tL\tL}(n)},
\end{align*}
where $r_{\tL \tL}(n)$ denotes the number of occurences of the block $\tL\tL$ in the representation of $n$ in base $2$. 
\end{Def}

The first values of the Rudin--Shapiro sequence are \[ 1, 1, 1, -1, 1, 1, -1, 1, 1, 1, 1, -1, -1, -1, 1, -1, 1, 1,\ldots \] and this is the sequence \href{https://oeis.org/A020985}{\underline{A020985}} in the \emph{On-line Encyclopedia of Integer Sequences} (OEIS)~\cite{oeis}. The sequence $\mathcal{R}$ was originally studied for its properties concerning the $L^2$ norm. We define the $L^2$ norm of a function $f$ by \[\left \lVert f\right \rVert_2=\left(\frac{1}{2\pi} \int_{0}^{2\pi} |f(t)|^2 \right)^{1/2} .\] For every sequence $(a_n)_{n\geq 0}$ over $\{-1,1\}$, we have \[ \sup \limits_{\theta \in \mathbb{R}} \left| \sum \limits_{0\leq n <N} a_ne^{in\theta} \right| \geq \left \lVert\sum \limits_{0\leq n <N} a_ne^{in\theta} \right \rVert_2=\sqrt{N} .\] Furthermore, almost every sequence $(a_n)_{n\geq 0}$ over $\{-1,1\}$ satisfies \[ \sup \limits_{\theta \in \mathbb{R}} \left| \sum \limits_{0\leq n <N} a_ne^{in\theta} \right| =O(\sqrt{N\log N}), \] see~\cite{SZ1954}. The Rudin--Shapiro is an example of a sequence with more cancelations than expected. Indeed, the sequence $\mathcal{R}$ satisfies the \emph{square-root property}, that is, there exists a constant $C>0$ such that for all $N\geq 0$ we have \[\sup \limits_{\theta \in \mathbb{R}} \left| \sum \limits_{0\leq n <N} r'_ne^{in\theta} \right| \leq C \sqrt{N}. \] Moreover, the power of $N$ is optimal and the best known constant is $C=(2+\sqrt{2})\sqrt{3/5}$, see~\cite[Page 122]{AS2003}. It is conjectured, and still, an open problem, that the optimal constant is $C=\sqrt{6}$. 

\smallskip

The Rudin--Shapiro $\mathcal{R}$ sequence is a famous automatic sequence, see~\cite{AS2003} for more details on this type of sequence. More precisely, $\mathcal{R}$ is a $2$-automatic sequence, and this means that the $n$th term of the sequence is computed by a deterministic finite automaton with output (DFAO) after reading the representation of $n$ in base $2$.

\begin{figure}[H]
\begin{center}
\begin{tikzpicture} [node distance = 3cm, on grid, every loop/.style={-stealth}, every initial by arrow/.style = {-stealth},]
 
    \node (q0) [state, initial, initial text=] {$q0/+1$};
    \node (q1) [state, right = of q0] {$q1/+1$};
    \node (q2) [state, right = of q1] {$q2/-1$};
    \node (q3) [state, right = of q2] {$q3/-1$};
    
    \path [-stealth]
    (q0) edge [bend left] node [above] {$\tL$}   (q1)
    (q0) edge [loop above] node {$\tO$}()
    
    (q1) edge [bend left] node [below] {$\tO$}   (q0)
    (q1) edge [bend left] node [above] {$\tL$}   (q2)
    
    (q2) edge [bend left] node [above] {$\tO$}   (q3)
    (q2) edge [bend left] node [below] {$\tL$} (q1)
    
    (q3) edge [bend left] node [below] {$\tL$}   (q2)
    (q3) edge [loop above] node {$\tO$}();
 
\end{tikzpicture}
\caption{DFAO generating the Rudin--Shapiro sequence $\mathcal{R}$.}
\end{center}
\end{figure}
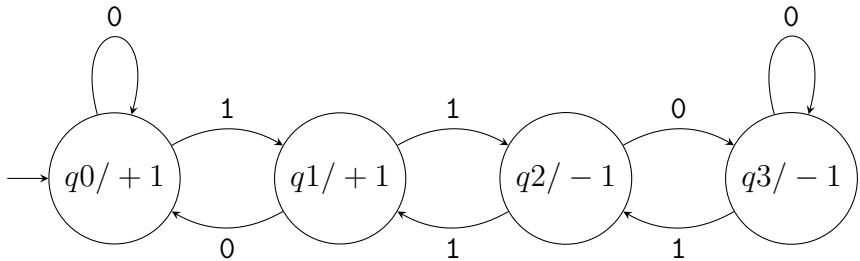

Automatic sequences are ``bad'' pseudorandom sequences with respect to a number of diffrent complexity measures. For instance, a classical measure of complexity is the subword complexity $p_{\mathcal{S}}(k)$, that is, the number of different patterns of length $k$ that appear in the whole sequence. Over a binary alphabet, the expected order of the subword complexity of a truly random binary sequence is $2^k$. However, it is known that the subword complexity of an automatic sequence is bounded by a linear function, see~\cite[Chapter 10]{AS2003} for more details. Notice that for the Rudin--Shapiro sequence, we also have an exact formula and $p_{\mathcal{R}}(k)=8k-8$, for all $k \geq 8$, see~\cite{AS1993}. Furthermore, the correlation of order $2$ of automatic sequences is large, see~\cite{MW2018}. 

\begin{theorem}[\cite{MW2018}] \label{corr2auto}
Let $\mathcal{S}$ be a $q$-automatic sequence. We have \[C_2(\mathcal{S},N) \gg N, \] where the constant depends on $q$ and the number of states of the underlying automaton.
\end{theorem}

Notice that the correlation of order $2$ of the Rudin--Shapiro sequence was already studied in~\cite{MS1998}, where the authors proved \[C_2(\mathcal{R},N) \geq \frac{1}{6}N, \] which can be compared to \Cref{corr2auto}. Nevertheless, the Rudin--Shapiro sequence is often cited for having a pseudorandom behavior regarding the correlation of order $2$ over small ranges. Indeed, Mauduit and S\'{a}rk\"{o}zy~\cite{MS1998} proved the following theorem. 

\begin{theorem}[\cite{MS1998}]
Let $d,N$ be positive integers and suppose that $d=o(N)$. Then \[ \sum_{n<N} r_nr_{n+d}=o(N). \]
\end{theorem} 

This behavior is radically different for larger orders of correlations, such as order $4$, even for a small range, see~\cite{MS1998}. For instance, we have the following result for consecutive terms. 

\begin{theorem}[\cite{MS1998}]
Let $M$ be a positive integer. Then \[ \sum_{n<2^M} r_nr_{n+1}r_{n+2}r_{n+3}=2^{M-1}. \]
\end{theorem} 

Therefore, these considerations on the correlations imply, once more, that the Rudin--Shapiro sequence is far from being a pseudorandom sequence. 

\subsection{Purpose of this paper}

The present paper is in the continuation of investigating the pseudorandomness of automatic sequences along subsequences. We refer to the survey~\cite{MW2022} for more details of pseudorandomness arising from automatic sequences. For instance, there is a large amount of literature on the subsequences of the famous Thue--Morse sequence. In particular, Drmota, Mauduit, and Rivat~\cite{DMR2019} have proved that the Thue--Morse sequence along squares is a normal sequence, that is, any fixed pattern of length $k$ appears with the same uniform frequency $1/2^k$, which is the expected behavior for truly random binary sequences. Moreover, M\"{u}llner~\cite{Mullner2018} extends this result to many automatic sequences, including the Rudin--Shapiro sequence.

\smallskip

The mentioned articles deal with polynomial subsequences, and in this paper, we investigate the pseudorandomness of this kind of sequence along the subsequence of prime numbers. Since there are more prime numbers than perfect squares, this problem seems easier in appearance. However, we do not have a simple explicit formula to express all the primes numbers, leading to different difficulties. For example, proving that the Thue--Morse sequence is normal along primes is still an open problem.  

\smallskip 

The main goal of this paper is to study sums of the form \begin{align*}
&\mathcal{S}_{k}(N)=\sum \limits_{p\leq N} r_p r_{p+1}\cdots r_{p+k}=\sum \limits_{p\leq N} (-1)^{r_{\tL\tL}(p)+\cdots+r_{\tL\tL}(p+k)},\\
&\mathcal{U}_{k}(N)=\sum \limits_{p\leq N} r_p r_{p+k}=\sum \limits_{p\leq N} (-1)^{r_{\tL\tL}(p)+r_{\tL\tL}(p+k)},
\end{align*} where $k\geq 1$ is an integer. The method used by Mauduit and S\'{a}rk\"{o}zy~\cite{MS1998} for this type of sums is not compatible with sums over primes, and we shall use exponential sums instead. Thereby we define the exponential sum $S_{\alpha}(N)$ by \begin{align} \label{def_exp_sum}
S_{\alpha}(N)=\sum \limits_{p\leq N} e\left(\alpha_0 r_{\tL \tL}(p) + \cdots + \alpha_k r_{\tL \tL}(p+k) \right)
\end{align} where $k\geq 1$ and $\alpha=(\alpha_0,\ldots,\alpha_k) \in \mathbb{R}^{k+1}$. Then we recover the sums $\mathcal{S}_{k}(N)$ and $\mathcal{U}_{k}(N)$ for specific vectors $\alpha$ with
\begin{align*}
&\mathcal{S}_{k}(N)=\mathcal{S}_{\alpha}(N), \text{ for } \alpha=(1/2,\ldots,1/2), \\
&\mathcal{U}_{k}(N)=\mathcal{S}_{\alpha}(N), \text{ for } \alpha=(1/2,0,\ldots,0,1/2).
\end{align*} The motivation of this work also comes from the study of the sum \[\sum \limits_{p\leq N} (-1)^{s_2(p)+\cdots+s_2(p+k)}, \quad \sum \limits_{p\leq N} (-1)^{s_2(p)+s_2(p+k)}\] from Aloui, Mauduit and Mkaouar~\cite{AMM2021}, where $s_2$ is the sum of digits function in base $2$. These authors have proved the following result. 

\begin{theorem}[\cite{AMM2021}]\label{sum of digits correlation}
For any nonnegative integer $k$, we have \begin{align*}
\lim \limits_{N\rightarrow +\infty} \frac{1}{\pi(N)}\sum \limits_{p\leq N} (-1)^{s_2(p)+\cdots+s_2(p+k)} =\lim \limits_{N\rightarrow +\infty} \frac{1}{N}\sum \limits_{ \substack{ n\leq N \\ n \text{odd}}} (-1)^{s_2(n)+\cdots+s_2(n+k)},
\end{align*}
and \begin{align*}
\lim \limits_{N\rightarrow +\infty} \frac{1}{\pi(N)}\sum \limits_{p\leq N} (-1)^{s_2(p)+s_2(p+k)} =\lim \limits_{N\rightarrow +\infty} \frac{1}{N}\sum \limits_{ \substack{ n\leq N \\ n \text{ odd}}} (-1)^{s_2(n)+s_2(n+k)}.
\end{align*}
\end{theorem} 

In this paper, we extend \Cref{sum of digits correlation} to the block digital function $r_{\tL\tL}$ and prove that the corresponding sums are equal to $0$ in the limit. This study is more challenging, and some of the methods used in~\cite{AMM2021} need to be specifically adapted to this case. 

\smallskip

This paper is structured as follows. In \Cref{section_result}, we give the main results of this paper. Some preliminaries are given in \Cref{Preliminaries}, and Sections \ref{section_thm1} and \ref{section_thm2} are dedicated to the proof of the main results of this paper.  

\section{Results}\label{section_result}

The main theorem is the following one. 

\begin{theorem} \label{main_thm}
We have for any $k\geq 1$ \begin{align} \label{main_thm1}
&\lim \limits_{N\rightarrow +\infty} \frac{1}{\pi(N)}\sum \limits_{p\leq N} (-1)^{r_{\tL\tL}(p)+\cdots+r_{\tL\tL}(p+k)}=0. \\  \label{main_thm2}
&\lim \limits_{N\rightarrow +\infty} \frac{1}{\pi(N)}\sum \limits_{p\leq N} (-1)^{r_{\tL\tL}(p)+r_{\tL\tL}(p+k)}=0. 
\end{align}
\end{theorem}

This theorem implies that correlations of order $k$ of the Rudin--Shapiro sequence are smaller along primes and closer to random behavior. Therefore, this new result proves that this sequence is far more random than the original sequence along the sequence of primes. 

\smallskip

The proof will be split into two parts, according to the parity of $k$. If $k$ is even, the following theorem is sufficient and provides an upper bound for the sum $\mathcal{S}_{\alpha}(N)$ for any $\alpha$. 

\begin{theorem}\label{First thm}
Let $k \in \mathbb{N}$,  $\alpha=(\alpha_0,\ldots,\alpha_k) \in \mathbb{R}^{k+1}$. Then for any integer $N$ large enough we have \begin{align*}
\mathcal{S}_{\alpha}(N)\ll (\log N)^KN^{1-\sigma}
\end{align*} where $\sigma=\min(1/2,c \lVert\tilde{\alpha_0}\rVert^2)$, the constant $K$ depends only on $k$ and the constant $c$ is independant from $\alpha$. 
\end{theorem} 

It is clear that if $\tilde{\alpha_0} \in \mathbb{Z}$,  the previous bound is trivial, and we shall prove a result for this case to extend \Cref{sum of digits correlation} to the block digital function $r_{\tL\tL}$. This case will be done more precisely in \Cref{section_thm2}. 

\section{Preliminaries} \label{Preliminaries}

In this section, we discuss the main differences between the Rudin--Shapiro sequence and the sum of digits function. The following lemma is the proof's key piece since it relies on two consecutive values of the Rudin--Shapiro sequence. Such as the function $r_{\tL\tL}$, we denote $r_{\tO\tL}(n)$ the number of occurrences of the block $\tO\tL$ in the $2$-ary expansion of $n$, including the block with the most significant digit. 

\begin{lem}\label{Legendre r11}
For any integer $n$, we have \begin{align*}
r_{\tL \tL}(n+1)=r_{\tL \tL}(n)+1-v_2(n+1)+r_{\tO \tL}(n)-r_{\tO \tL}(n+1).
\end{align*}
\end{lem}

\begin{proof}
Let $n$ be an integer and $n=\sum_{i=0}^{\ell} \varepsilon_i2^{i}$ be its $2$-ary expansion. Since $r_{\tL\tL}(n)=\sum_{i=0}^{\ell} \varepsilon_{i+1}\varepsilon_{i}$, we have \begin{align*}
n-r_{\tL \tL}(n)&=\sum \limits_{i=0}^{\ell}\left(\varepsilon_i 2^{i}-\varepsilon_{i+1}\varepsilon_{i}\right)=\sum \limits_{i=0}^{\ell}\varepsilon_{i}\left(2^{i}-\varepsilon_{i+1}\right)\\&=\sum \limits_{i=0}^{\ell}\varepsilon_{i}\left(2^{i}-1\right)+\sum \limits_{\substack{i=0 \\ \varepsilon_{i+1}=0}}^{\ell}\varepsilon_{i}=v_2(n!)+r_{\tO \tL}(n).
\end{align*} Thus we have \begin{align*}
n&=r_{\tL \tL}(n)+v_2(n!)+r_{\tO \tL}(n) \\ n+1&=r_{\tL \tL}(n+1)+v_2((n+1)!)+r_{\tO \tL}(n+1)
\end{align*} which implies the result since $v_2((n+1)!)=v_2(n+1)+v_2(n!)$.
\end{proof}

By iterating \Cref{Legendre r11} we have \begin{align}\label{recursion}
r_{\tL\tL}(n+i)=r_{\tL\tL}(n)+i-\sum \limits_{j=1}^{i}v_2(n+j)+\Delta(n,i)
\end{align} where $\Delta(n,i)=r_{\tO\tL}(n)-r_{\tO\tL}(n+i)$, is the term that appear on the right-hand side in the previous lemma. 

\begin{remark}
This relation is very similar to the usual relation for the sum of digits function \begin{align*}
s_q(n+1)=s_q(n)+1-(q-1)v_q(n+1),
\end{align*} that can be proved by using Legendre's formula. We could also prove \Cref{Legendre r11} by using the fact that $s_2(n)=r_{\tL\tL}(n)+r_{\tO\tL}(n)$. 
\end{remark}

In the following lemma, we describe all the possible values of $\Delta(n,1)$. 

\begin{lem}\label{Delta}
Let $n$ be an integer. Then $\Delta(n,1) \in \{-1,0,1\}$. More precisely, let us denote $u=v_2(n+1)$ and $\varepsilon \in \{0,1\}$ such that $n+1\equiv 2^{u} +\varepsilon2^{u+1} \pmod {2^{u+2}}$. We have
\begin{itemize}
\item If $u=0$ and $\varepsilon=0$, $\Delta(n,1)=-1$. 
\item If $u=0$ and $\varepsilon=1$, $\Delta(n,1)=0$. 

\smallskip

\item If $u>0$ and $\varepsilon=0$, $\Delta(n,1)=0$. 
\item If $u>0$ and $\varepsilon=1$, $\Delta(n,1)=1$. 
\end{itemize}

\end{lem}

\begin{proof}
We study $\Delta(n,1)=r_{\tO\tL}(n)-r_{\tO\tL}(n+1)$ in more detail. We have the two following cases: \begin{itemize}
\item If $v_2(n+1)=0$,  we have \begin{itemize}
\item If $n+1\equiv 1 \pmod 4$,  we have $(n)_2=\omega\tO\tO$ and $(n+1)_2=\omega\tO\tL$ for some $\omega$ a binary word. Then $\Delta(n,1)=-1$.
\item If $n+1\equiv 3 \pmod 4$,  we have $(n)_2=\omega\tL\tO$ and $(n+1)_2=\omega\tL\tL$ for some $\omega$ a binary word. Then $\Delta(n,1)=0$.
\end{itemize}
\item If $v_2(n+1)=u>0$,  we have \begin{itemize}
\item If $n+1\equiv 2^{u} \pmod {2^{u+2}}$ ie $(n)_2=\omega\tO\tO\tL\cdots\tL$,  with a block $u$ $\tL$-bits. Thus $(n+1)_2=\omega\tO\tL\tO\cdots\tO$, for some $\omega$ a binary word, and $\Delta(n,1)=0$. 
\item If $n+1\equiv 2^{u} +2^{u+1} \pmod {2^{u+2}}$ ie $(n)_2=\omega\tL\tO\tL\cdots\tL$,  with a block $u$ $\tL$-bits. Thus $(n+1)_2=\omega\tL\tL\tO\cdots\tO$, for some $\omega$ a binary word, and $\Delta(n,1)=1$. 
\end{itemize}
\end{itemize}
\end{proof}

\begin{remark}
The dependance of $\Delta(n,1)$ from $v_2(n+1)$ is in fact only a dependance on $v_2(n+1)=0$ or not, i.e. the first digit in base $2$.
\end{remark}

As a direct consequence of this lemma, we have that $\Delta(n,i)\in \{-i,\ldots,i\}$ and it depends only on the $2$-adic valuations of $u_j=v_2(n+j)$ for $1\leq j \leq i$ and on the vector $\varepsilon=(\varepsilon_1,\ldots,\varepsilon_i)$ such that $n+j\equiv 2^{u_j}+\varepsilon_j2^{u_j+1} \pmod{2^{u_j+2}}$. We shall see later that $\Delta(n,i)$ belongs to a much smaller set. 

\smallskip

The following fundamental fact allows us to detect congruences with exponential sums and to study arithmetic properties with analytic tools. 

\begin{lem}  \label{lemme_classique_sum_exp}
Let $a \in \mathbb{Z}$. We have \[\sum \limits_{x=0}^{m-1} e\left( \frac{ax}{m} \right)=\left\{
    \begin{array}{ll}
        m, & \text{si } a \equiv 0 \pmod m,\\
        0, & \text{sinon.}
    \end{array}
\right.\]
\end{lem}

The following lemma is the Prime Number theorem for arithmetic progressions, see \cite{Apostol1998}[Theorem 7.3]. 

\begin{lem}[Prime Number Theorem for arithmetic progressions]\label{PNTAP}
Let $q\geq 1$ and $a\geq 0$ such that $(a,q)=1$. For all $N\geq 2$, we have \begin{align*}
\pi(N;q,a)=\frac{\pi(N)}{\varphi(q)}(1+o(1))=\frac{1}{\varphi(q)}\frac{x}{\log x}(1+o(1)), \qquad N\rightarrow +\infty.
\end{align*}
\end{lem}

To prove the main theorem, we will study congruence systems of the form \begin{align}\label{congruence system}
\left\{ \begin{array}{ll} 
p+1\equiv 2^{u_1}+\varepsilon_1 2^{u_1+1} \pmod {2^{u_1+2}} \\
\; \; \vdots \\
p+k\equiv 2^{u_k}+\varepsilon_k 2^{u_k+1} \pmod {2^{u_k+2}}  \\
\end{array}
\right. 
\end{align} for some integers $(u_i)$. By the Chinese Remainder Theorem, we have the following lemma. 

\begin{lem} \label{CRT}
Let $\varepsilon=(\varepsilon_1,\ldots,\varepsilon_k)\in \{0,1\}^k$ and $(u_j)_j$ a sequence such that $u_{2j+1}$ is a positive integer and $u_{2j}=0$ for $j\geq 1$. The following system \begin{align*}
p+i\equiv 2^{u_i}+\varepsilon_i 2^{u_i+1} \pmod {2^{u_i+2}}, \; \forall i,
\end{align*} has a solution if and only if \begin{align*}
2^{u_j}+\varepsilon_j2^{u_j+1}-j\equiv 2^{u_{\ell}}+\varepsilon_{\ell}2^{u_{\ell}+1}-\ell \pmod{2^{\min(u_j,u_{\ell})+2}}, \; \forall j\neq \ell,
\end{align*} and the solution belongs to a certain class modulo $2^{\max_{i}(u_i)+2}$. 
\end{lem} 

\begin{proof}
The Chinese Remainder Theorem for non-coprime moduli states that the system of congruence $x \equiv a_i \pmod a_i$ for $1\leq i \leq r$ has a solution if and only if for each pair $(j,\ell)$, $a_j\equiv a_{\ell} \pmod{\gcd(a_j,a_{\ell})}$. The solution is unique modulo $\lcm(m_1,\ldots,m_r)$. We apply this theorem for $a_i=2^{u_i}+\varepsilon_i2^{u_i+1}-i$ and $m_i=2^{u_{i}+2}$ for all $i\geq 0$. Therefore the lemma is proved sicne for a pair $(j,\ell)$ we have $\gcd(2^{u_j+2},2^{u_{\ell}+2})=2^{\min(u_j,u_{\ell})+2}$ and $\lcm(2^{u_1+2},\ldots,2^{u_k+2})=2^{\max_i(u_i)+2}$. 
\end{proof}

We also have the following lemma~\cite[Lemma 4.1]{AMM2021}, adapted for $q=2$ in this paper. 

\begin{lem}[{\cite[Lemma 4.1]{AMM2021}}]\label{lemma 4.1}
Let $k \in \mathbb{N}$, $p>2$ a prime number. Let $(u_j)_j$, for odd $j$, a sequence of positive integers. Then \[2^{u_j} || p+j, \forall j \Leftrightarrow \left\{ \begin{array}{ll} 
u_i = \max(u_j) > \lfloor \log_2(k-1) \rfloor, \\
2^{u_j} || p+j, \\
u_j=v_2(|i-j|) \text{ for all } j \neq i.
\end{array}
\right.\]
\end{lem}

This lemma implies that for such a system of congruences as \eqref{congruence system}, only a single $u_j$ is not entirely determined and can vary. 

\smallskip

For sums over prime numbers, it is a standard procedure to introduce the corresponding exponential sums over the integers with weight $\Lambda(n)$. Notice that $r_{\tL\tL}$ is not a \emph{digital function} but rather a \emph{block--digital function} since it depends on blocks of digits of length $2$. Therefore, estimates for digital functions of~\cite{MMR2014} used in~\cite{AMM2021} are no longer useful and we shall use~\cite[Theorem 4]{MR2015} instead.

\begin{lem}[{\cite[Theorem 4]{MR2015}}]
For any integer $N \geq 2$ and for all $(\alpha,\beta)$, we have \begin{align*}
\sum \limits_{n\leq N} \Lambda(n)e(\alpha r_{\tL \tL}(n)+\beta n) \ll \log(N)^{11/4}N^{1-c||\alpha||^2}.
\end{align*}
\end{lem}


\section{Proof of \texorpdfstring{\Cref{First thm}}{Lg}} \label{section_thm1}

We first introduce some lemmas to prove \Cref{First thm}.

\begin{lem} \label{lemma1}
Let $\alpha,\beta \in \mathbb{R}$ and $(y_n)_{n \in  \mathbb{N}}$ be a sequence of complex numbers. Then for any integer $i\geq 1$ and for any integer $N$ large enough, we have \begin{align*}
\sum \limits_{n\leq N}y_n e(\alpha v_2(n+i)+\beta\Delta(n,i)) \ll_{i}(\log N)^i\max_{r\in [0,1)} \left| \sum \limits_{n\leq N}y_ne(rn) \right| .
\end{align*} 
\end{lem}

The main difference with~\cite[Lemma 3.1]{AMM2021} is that we have to deal with more congruence systems, implying that we have the power of $\log N$ instead. 

\begin{proof}
By~\Cref{Delta}, a system of congruences fully determines the value of $\Delta(n,i)$. Let $M=\lfloor \log_2(N+i) \rfloor$. We write the sum \begin{align*}
S=\sum \limits_{n\leq N}y_n e(\alpha v_2(n+i)+\beta\Delta(n,i))
\end{align*} in the following way \begin{align*}
S=\sum \limits_{(u_1,\ldots,  u_i) \in [M]^{i}} \sum \limits_{\substack{n\leq N \\ v_2(n+j)=u_j \\ 1\leq j\leq i}} y_n e(\alpha u_i +\beta \Delta(n,i)).
\end{align*} Since $v_2(n+1)=u_1$ if and only if $ n+1 \equiv 0  \pmod {2^{u_1}}$ and $n+1 \not \equiv 0~(\Mod 2^{u_1+1})$, we have \[S=S_0-S_1,\] where \begin{align*}
S_0&=\sum \limits_{(u_1,\ldots,  u_i) \in [M]^i } \sum \limits_{\substack{n\leq N \\ v_2(n+j)=u_j \\ 2\leq j\leq i}}\sum \limits_{n+1 \equiv 0~(\Mod {2^{u_1}})}y_n e(\alpha u_i +\beta \Delta(n,i)), \\ S_1&=\sum \limits_{(u_1,\ldots,  u_i) \in [M]^i } \sum \limits_{\substack{n\leq N \\ v_2(n+j)=u_j \\  2\leq j\leq i}}\ \sum \limits_{n+1 \equiv 0~(\Mod {2^{u_1+1}})}y_n e(\alpha u_i +\beta \Delta(n,i)).
\end{align*} Then we apply the same cutting on $S_0$ and $S_1$ according to the values of $v_2(n+2)=u_2$. By applying this cutting on $v_2(n+j)$ for every $j\geq 1$, we have \begin{align*}
S=T_0+T_1+\cdots+T_{2^{i}-1},
\end{align*} where $T_{\lambda}$, $0\leq \lambda < 2^{i}$, is of the form \begin{align*}
T_{\lambda}=(-1)^{\tilde{\varepsilon_1}}\sum \limits_{(u_1,\ldots,  u_i) \in [M]^i } \sum \limits_{ \substack{n\leq N \\ n+j \equiv 0~(\Mod 2^{u_j+\varepsilon_j}) \\  1\leq j \leq i}} y_n e(\alpha u_i+\beta\Delta(n,i) )
\end{align*} with  $\lambda=\sum_{1\leq j \leq i}\varepsilon_j2^{j-1}$, for some $\varepsilon=(\varepsilon_1,\ldots,\varepsilon_i) \in \{0,1\}^{i}$ and let us recall that $\tilde{\varepsilon_1}$ is defined by $\tilde{\varepsilon_1}=\varepsilon_1+\cdots+\varepsilon_i$. By~\Cref{Delta}, the value $\Delta(n,i)$ only depends of $\lambda$ in the inner sum of $T_{\lambda}$, denoted by $\Delta_{\lambda}$. Thus, we have \begin{align*}
T_{\lambda}=(-1)^{\tilde{\varepsilon_1}}\sum \limits_{(u_1,\ldots,  u_i) \in [M]^i }e(\alpha u_i+\beta \Delta_{\lambda}) \sum \limits_{ \substack{n\leq N \\ n+j \equiv 0~(\Mod 2^{u_j+\varepsilon_j}) \\ 1\leq j \leq i}} y_n. 
\end{align*} By~\Cref{lemme_classique_sum_exp}, we can translate the set of solutions of the following system of congruences \[\left\{
    \begin{array}{ll}
    	&n+1\equiv 0 \pmod {2^{u_1+\varepsilon_1}},\\
    	&\; \; \vdots \\ 
    	&n+i\equiv 0 \pmod {2^{u_i+\varepsilon_i}},
    \end{array}
\right. \] by an exponential sum. Then we have \begin{align*}
T_{\lambda}&=(-1)^{\tilde{\varepsilon_1}}\sum \limits_{(u_1,\ldots,  u_i) \in [M]^i }e\left( \alpha u_i+\beta \Delta_{\lambda} \right) \\ & \qquad \qquad \times \sum \limits_{n\leq N} y_n \left(\frac{1}{2^{u_1+\varepsilon_1}} \sum \limits_{0\leq \ell_1 <2^{u_1+\varepsilon_1}}e\left(\frac{\ell_1(n+1)}{2^{u_1+\varepsilon_1}} \right)\right) \\ &\qquad \qquad \times \cdots \times \left(\frac{1}{2^{u_i+\varepsilon_i}} \sum \limits_{0\leq \ell_i <2^{u_i+\varepsilon_i}}e\left(\frac{\ell_i(n+i)}{2^{u_i+\varepsilon_i}}\right)\right) \\
&=(-1)^{\tilde{\varepsilon_1}}\sum \limits_{(u_1,\ldots,  u_i) \in [M]^i } e\left(\alpha u_i+\beta \Delta_{\lambda}\right) \frac{1}{2^{u_1+\cdots+u_i+\tilde{\varepsilon_1}}} \\ & \qquad \qquad \times \sum \limits_{ \substack{0\leq \ell_j < 2^{u_j+\varepsilon_j}\\ 1\leq j\leq i}} e\left(\frac{\ell_1}{2^{u_1+\varepsilon_1}}+\cdots+ \frac{i\ell_i}{2^{u_i+\varepsilon_i}} \right) \\ & \qquad \qquad \times \sum \limits_{n\leq N} y_n e\left(\left(\frac{\ell_1}{2^{u_1+\varepsilon_1}}+\cdots+ \frac{\ell_i}{2^{u_i+\varepsilon_i}}\right)n\right).
 \end{align*} Thus for every $0\leq \lambda <2^{i}$, we have \begin{align*}
T_{\lambda} \ll M^i \max \limits_{r\in [0,i)} \left| \sum \limits_{n\leq N}y_ne(rn) \right|= M^i \max \limits_{r\in [0,1)} \left| \sum \limits_{n\leq N}y_ne(rn) \right|,
\end{align*} by periodicity of the function $e(x)$.
\end{proof}

\begin{lem}\label{estimate1}
Let $k\geq 2$, $(\alpha_1,\ldots,\alpha_k)\in \mathbb{R}^k$, $(\beta_1,\ldots,\beta_k)\in \mathbb{R}^k$ and $(y_n)_{n \geq 0}$ a sequence of complex numbers. Then for any fixed $k \in \mathbb{N}$ and for any integer $N$ large enough we have \begin{align*}
\sum \limits_{n\leq N}y_n e\left (\sum \limits_{1\leq i \leq k}\alpha_i v_2(n+i)+\sum \limits_{1\leq i \leq k}\beta_i\Delta(n,i)\right) \\ 
\ll_{k}(\log N)^K \max_{r\in [0,1)} \left| \sum \limits_{n\leq N} y_ne(rn) \right|
\end{align*} with $K=k(k+1)/2$. 
\end{lem}

\begin{proof}
By \Cref{lemma1}, we have \begin{align*}
\sum \limits_{n\leq N}&y_n e\left (\sum \limits_{1\leq i \leq k}\alpha_i v_2(n+i)+\sum \limits_{1\leq i \leq k}\beta_i\Delta(n,i)\right)\\&=\sum \limits_{n\leq N}y_n e\left (\sum \limits_{1\leq i \leq k-1}\alpha_i v_2(n+i)+\sum \limits_{1\leq i \leq k-1}\beta_i\Delta(n,i)\right)e\left (\alpha_k v_2(n+k)+\beta_k\Delta(n,k)\right), \\ & \ll_{k}(\log N)^K \max_{r\in [0,1)} \left| \sum \limits_{n\leq N} y_ne(rn)e\left(\sum \limits_{1\leq i \leq k-1}\beta_i\Delta(n,i)\right) \right|. 
\end{align*}
Then we get the desired bound by applying $k$ times the precedent lemma. 
\end{proof}

We are now ready to prove \Cref{First thm}.  

\begin{proof}[Proof of \Cref{First thm}]
Set \begin{align*}
\Psi(N)=\sum \limits_{n\leq N}\Lambda(N)e\left(\alpha_0 r_{\tL \tL}(n) + \cdots + \alpha_k r_{\tL \tL}(n+k) \right).
\end{align*} From \eqref{recursion} we obtain \begin{align}\label{recursion_bis}
\sum \limits_{i=0}^k \alpha_i r_{\tL\tL}(n+i) = \tilde{\alpha} + \tilde{\alpha_0}r_{\tL\tL}(n) - \sum \limits_{i=1}^k \tilde{\alpha_i}v_2(n+i) +\sum \limits_{i=1}^k \alpha_i \Delta(n,i).
\end{align}
Then by \Cref{estimate1} and \eqref{recursion_bis} with $y_n=\Lambda(N)e(\tilde{\alpha_0}r_{\tL\tL}(n))$ for all $n\leq N$, we have \begin{align*}
\Psi(N)&=e(\tilde{\alpha}) \sum \limits_{n\leq N} \Lambda(N)e\left(\tilde{\alpha_0}r_{\tL\tL}(n) - \sum \limits_{i=1}^k \tilde{\alpha_i}v_2(n+i) +\sum \limits_{i=1}^k \alpha_i \Delta(n,i)\right). \\ & \ll_k (\log N)^K \max \limits_{r \in [0,1)}\left| \sum \limits_{n\leq N}\Lambda(N)e(\tilde{\alpha_0}r_{\tL\tL}(n)+rn) \right|.  \\ & \ll_k (\log N)^{K+11/4}N^{1-c||\tilde{\alpha_0}||^2}.
\end{align*} Thus by partial summation, \begin{align*}
\mathcal{S}_{\alpha}(N)\ll \frac{1}{\log N} \max \limits_{m\leq N} |\Psi(m)| + \sqrt{N}, 
\end{align*} and the theorem is proved. 
\end{proof}


\section{Proof of \texorpdfstring{\Cref{main_thm}}{Lg}}  \label{section_thm2}

In this section, we prove \Cref{main_thm}, and we start by establishing \eqref{main_thm1}, the first statement of the theorem. The proof of this case is split into two cases, according to the parity of $k$. The case $k$ odd is the most difficult since \Cref{First thm} is insufficient. 

\subsection{First statement of \Cref{main_thm} and $k$ even}

Let $k\geq 1$ be even and $\alpha=(1/2,\ldots,1/2) \in \mathbb{R}^{k+1}$. 

\begin{proof}
We have $\tilde{\alpha_0}=\frac{k+1}{2}$ and $\lVert \tilde{\alpha_0} \rVert =\frac{1}{2}$. Thus, by \Cref{First thm}, we have \begin{align*}
\mathcal{S}_{k}(N)=\mathcal{S}_{\alpha}(N) \ll (\log N)^K N^{1-\sigma},
\end{align*} for $\sigma=\min(1/2,c/4)$ with $c>0$. Thus, by \Cref{PNTAP}, we have \begin{align*}
\lim \limits_{N\rightarrow +\infty} \frac{1}{\pi(N)}\sum \limits_{p\leq N} (-1)^{r_{\tL\tL}(p)+\cdots+r_{\tL\tL}(p+k)}=0. 
\end{align*} 
\end{proof}

\subsection{First statement of \Cref{main_thm} and $k$ odd}

Let $k\geq 1$ be odd and $\alpha=(\alpha_0,\ldots,\alpha_k) \in \mathbb{R}^{k+1}$ be such that $\tilde{\alpha_0}=h \in \mathbb{Z}$. 

\smallskip

Let us denote $\Lambda_k$ the set of all possible values of $\Delta_k(p)=(\Delta(p,1),\ldots, \Delta(p,k))$. Since $\Delta_k(p)$ is fully determined by the values of $(\Delta(p,1),\Delta(p+1,1),\ldots,\Delta(p+k-1,1))$, and each term has two possible values, we have $\Card(\Lambda_k)=2^k$. Let us denote \begin{align*}
V_{b,a}(N)=\sum \limits_{p\leq N}e\left(\sum \limits_{i=1}^{k}b_iv_2(p+i)+\sum \limits_{i=1}^{k}a_i\Delta(p,i)\right)
\end{align*}
for real vectors $b=(b_1,\ldots,b_k)$ and $a=(a_1,\ldots,a_k)$. 

\begin{proof}
Let us recall that \begin{align*}
\sum \limits_{i=0}^k \alpha_i r_{\tL\tL}(n+i) = \tilde{\alpha} + \tilde{\alpha_0}r_{\tL\tL}(n) - \sum \limits_{i=1}^k \tilde{\alpha_i}v_2(n+i) +\sum \limits_{i=1}^k \alpha_i \Delta(n,i).
\end{align*} Therefore, we have \begin{align*}
S_{\alpha}(N)&=\sum \limits_{p\leq N} e\left(\sum \limits_{i=0}^k \alpha_i r_{\tL\tL}(n+i) \right), \\ 
&=\sum \limits_{p\leq N} e\left( \tilde{\alpha} + \tilde{\alpha_0}r_{\tL\tL}(n) - \sum \limits_{i=1}^k \tilde{\alpha_i}v_2(n+i) +\sum \limits_{i=1}^k \alpha_i \Delta(n,i) \right), \\
&= e( \tilde{\alpha} ) V_{b,\alpha'}(N),
\end{align*} for $b=(\tilde{\alpha_1},\ldots,\tilde{\alpha_k})$ and $\alpha'=(\alpha_1,\ldots,\alpha_k)$ since $e(\tilde{\alpha_0}r_{\tL\tL}(n))=1$ by the hypothesis $\tilde{\alpha_0}=h \in \mathbb{Z}$. In the following, we denote $\alpha'$ by $\alpha$ since only $\alpha_0$ is omitted. Therefore we now focus on the study of $V_{b,\alpha}(N)$. 

\smallskip

Let $M=\log _2(N+k)$, thus we have \begin{align} \label{Vba}
V_{b,\alpha}(N)=\sum \limits_{\Delta \in \Lambda_k} e(\alpha \cdot \Delta)\sum \limits_{\substack{0\leq u_i \leq M \\ 1\leq i \leq k}} e(b\cdot u)\sum \limits_{\substack{p\leq N \\ 2^{u_i} || p+i \\ 1\leq i \leq k \\ \Delta_k(p)=\Delta}} 1 +O(1)
\end{align} since we can remove $p=2$ from the sum. 

\smallskip 

Let $\Delta=(\Delta_1,\ldots,\Delta_k) \in \Lambda_k$ and $0\leq u_i \leq M, 1\leq i \leq k$. The congruence system that appears in the inner sum of $\eqref{Vba}$ is \begin{align*} \left\{ \begin{array}{ll} 
\Delta(p,1)=\Delta_1, \; v_2(p+1)=u_1, \\
\; \; \vdots \\
\Delta(p,k)=\Delta_k, \; v_2(p+k)=u_k. \\
\end{array}
\right. 
\end{align*} This system can be traduced uniquely by a system of congruences for a vector $\varepsilon=(\varepsilon_1,\ldots,\varepsilon_k)\in \{0,1\}^k$ such that \begin{align}\label{congruence system_2}
\left\{ \begin{array}{ll} 
p+1\equiv 2^{u_1}+\varepsilon_1 2^{u_1+1} \pmod {2^{u_1+2}} \\
\; \; \vdots \\
p+k\equiv 2^{u_k}+\varepsilon_k 2^{u_k+1} \pmod {2^{u_k+2}}  \\
\end{array}
\right. 
\end{align} By \Cref{lemma 4.1} , a solution of \eqref{congruence system_2} implies that \[ \left\{ \begin{array}{ll} 
u_i = \max(u_j) > \lfloor \log_2(k-1) \rfloor, \\
2^{u_j} || p+j, \\
u_j=v_2(|i-j|) \text{ for all } j \neq i,
\end{array}
\right.\] for some odd $i$. 

\smallskip

Let us define $\varepsilon'=(\varepsilon_1,\ldots,\varepsilon_{i-1},1-\varepsilon_i, \varepsilon_{i+1},\ldots, \varepsilon_k)$, where $i$ is given by the system \eqref{congruence system_2} and \Cref{lemma 4.1}, and let $\Delta'$ be the associated vector of $\Lambda_k$. Notice that there does not necessarily exist a solution of the system \eqref{congruence system_2}. But in this case, both sums associated with $\Delta$ and $\Delta'$ will contribute $0$ to the final sum.

\smallskip

Without loss of generality, we can suppose that $\Delta(p+i,1)=0$ and $\Delta'(p+i,1)=1$ since we can exchange $\Delta$ and $\Delta'$,  otherwise. Therefore, we have \[\left\{ \begin{array}{ll} 
\Delta(p,1)=\lambda_1 \\
\vdots \\ 
\Delta(p+i,1)=0 \\
\vdots \\
\Delta(p+k-1,1)=\lambda_k \\
\end{array}
\right. \] for some $\lambda_{2j} \in \{0,1\}$, $\lambda_{2j+1} \in \{-1,0\}$, $1\leq 2j+1 \leq k$ and $2j+1\neq i$. Then we have \[ \Delta=\begin{pmatrix}
\lambda_1 \\
\vdots \\ 
\lambda_1 +\cdots +\lambda_{i-1}+0 \\ 
\vdots \\ 
\lambda_1 +\cdots +0+\cdots+\lambda_{k}
\end{pmatrix}. \] We also have \[\left\{ \begin{array}{ll} 
\Delta'(p,1)=\lambda_1 \\
\vdots \\ 
\Delta'(p+i,1)=1 \\
\vdots \\
\Delta'(p+k-1,1)=\lambda_k \\
\end{array}
\right.\] and \[\Delta'=\begin{pmatrix}
\lambda_1 \\
\vdots \\ 
\lambda_1 +\cdots +\lambda_{i-1}+1 \\ 
\vdots \\ 
\lambda_1 +\cdots +1+\cdots+\lambda_{k}
\end{pmatrix}. \] We can write the sum \begin{align} \label{final_sum}
V_{b,\alpha}(N)=\sum \limits_{(\Delta,\Delta') \in \Lambda_k^2}( e(\alpha \cdot \Delta)+e(\alpha \cdot \Delta'))\sum \limits_{\substack{0\leq u_i \leq M \\ 1\leq i \leq k}} e(b\cdot u) \frac{1}{\varphi(2^{u_i+2})} \pi(N)+O(1),
\end{align} where we have gathered $\Delta$ and $\Delta'$ since the solutions to the congruence system define a residue class with the same modulus by \Cref{CRT} and by \Cref{PNTAP}

\smallskip 

Let $\alpha=(1/2,\ldots,1/2)$ and we will now prove that the systems associated with $\Delta$ and $\Delta'$ sum to $0$ in the final sum. Indeed, we have \begin{align*} 
&e(\alpha \cdot \Delta)=e\left(\frac{1}{2} \sum \limits_{1\leq j \leq k} \sum \limits_{1\leq r \leq j} \lambda_r \right). \\
&e(\alpha \cdot \Delta')=e\left(\frac{1}{2} \sum \limits_{1\leq j \leq k} \sum \limits_{1\leq r \leq j} \lambda_r + \frac{1}{2}(k-i+1) \right)=-e(\alpha \cdot \Delta),
\end{align*} since $k$ and $i$ are odd. Therefore, we have $S_{\alpha}(N)=O(1)$ and the theorem is proved. 

\end{proof}

\subsection{Second statement of \Cref{main_thm}}

In this subsection, we prove \eqref{main_thm2} for all $k \in \mathbb{N}$. Let $k \in \mathbb{N}$ and $\alpha=(1/2,0,\ldots,0,1/2)$. Therefore, we have $ \lVert \tilde{\alpha_0} \rVert=0$, whether $k$ is odd or not, and \Cref{First thm} is insufficient. 

\begin{proof}
We have \begin{align*} 
&e(\alpha \cdot \Delta)=e\left( \frac{1}{2} (\lambda_1+\lambda_1+\cdots+\lambda_{k}) \right) \\ 
&e(\alpha \cdot \Delta)=e\left( \frac{1}{2} (\lambda_1+\lambda_1+\cdots+\lambda_{k}+1) \right)=-e(\alpha \cdot \Delta).
\end{align*} Therefore, by \eqref{final_sum}, we can prove \eqref{main_thm2} in a similar way as \eqref{main_thm1} for the odd case. 
\end{proof}

\begin{remark}
Notice that in the study of $\mathcal{S}_k(N)$, it was impossible to gather $k$ odd and $k$ even in the same proof. 
\end{remark}

\section{Related open problems}

We conclude this paper with two related open problems. 

\begin{problem}
Study the correlation of order $k$ of the Rudin--Shapiro sequence in the general case, not only for consecutive terms. 
\end{problem}

One can define an extension of the Rudin--Shapiro sequence to any fixed binary word $\omega$ in the following way. Let $e_{\omega}(n)$ be the number of occurrences of the binary $\omega$ in the binary expansion of $n$. Then the sequence $\mathcal{S}_{\omega}=((-1)^{e_{\omega}(n)})_{n\geq 0}$ is a $2$-automatic sequence and generalize the Rudin--Shapiro sequence since $\mathcal{R}=\mathcal{S}_{\tL\tL}$. 

\begin{problem}
Extend \Cref{main_thm} to sequence $\mathcal{S}_{\omega}$ for any fixed binary word $\omega$ of length $k\geq 2$. 
\end{problem}

\bibliographystyle{amsplain}
\bibliography{biblio}

\end{document}